\documentclass[12pt]{amsart}
\usepackage{amsmath,amsfonts,amssymb}
\numberwithin{equation}{section}

\newtheorem{theorem}{Theorem}[section]
\newtheorem{lemma}[theorem]{Lemma}
\newtheorem{corollary}[theorem]{Corollary}
\newtheorem{proposition}[theorem]{Proposition}

\newtheorem{definition}[theorem]{Definition}

\DeclareMathOperator{\Det}{Det} \DeclareMathOperator{\Tr}{Tr}
\DeclareMathOperator{\Ad}{Ad}

 \DeclareMathOperator{\Rea}{Re}

\DeclareMathOperator{\Diag}{Diag}

\title [Complex Weyl correspondence...]{ Complex Weyl correspondence
for a generalized diamond group}
\author {Benjamin Cahen}
\address{Universit\'e de Lorraine, Site de Metz, UFR-MIM,
D\'epartement de math\'ematiques,
B\^atiment A,
3 rue Augustin Fresnel, BP 45112,
57073 METZ Cedex 03, France.}
\email{benjamin.cahen@univ-lorraine.fr}

\subjclass[2000]{22E45; 22E70; 81R05; 81S10; 81R30.} \keywords{Complex Weyl calculus;
Weyl correspondence; Fock space; Bargmann-Fock representation; Berezin quantization; Heisenberg group; diamond group;
reproducing kernel Hilbert space.}

\begin{document}

\maketitle

\begin{abstract} The generalized diamond group is the semi-direct product $G$ of the abelian
group ${\mathbb R}^m$ by the $(2n+1)$-dimensional
Heisenberg group $H_n$. We construct the generic representations of $G$ on the Fock space by extending those of $H_n$. Then we study the Berezin correspondence and the complex Weyl correspondence in connection with a generic representation $\pi$ of $G$, proving in particular that these correspondences are covariant with respect to $\pi$. We give also some explicit formulas for the Berezin symbols and the complex Weyl symbols of the representation operators $\pi(g)$ for $g\in G$. These results are applied to recover various formulas involving the Moyal product. Moreover, we relate $\pi$ to a coadjoint orbit of $G$ in the spirit of the Kirillov-Kostant method of orbits. This allows us to establish that the complex Weyl correspondence is a Stratonovich-Weyl correspondence for $\pi$. 
\end{abstract}

\vspace{1cm}

\section {Introduction} \label{sec:intro}
The notion of Stratonovich-Weyl correspondence was
introduced in \cite{St} in order to quantize homogeneous spaces as, for instance,  coadjoint orbits of Lie groups. 
Stratonovich-Weyl correspondences were systematically studied  by J.M.
Gracia-Bond\`{i}a, J.C. V\`{a}rilly and various collaborators, see
\cite{GB, GVS} and references therein.  

\begin{definition} \label{def1} \rm{\cite{GB}} Let $G$ be a Lie group and $\pi$ a unitary representation
of $G$ on a Hilbert space $\mathcal H$. Let $M$ be a homogeneous
$G$-space and let $\mu$ be a (suitably normalized) $G$-invariant
measure on $M$. Then a Stratonovich-Weyl correspondence for the
triple $(G,\pi, M)$ is an isomorphism $W$ from a vector space of
operators on $\mathcal H$ to a space of (generalized) functions on
$M$  satisfying the following properties:

\begin{enumerate}

\item $W$ maps the identity operator of $\mathcal H$ to the constant function $1$;

\item the function $W(A^{\ast})$ is the complex-conjugate of $W(A)$;

\item Covariance: we have $W(\pi (g)\,A\,\pi (g)^{-1})(x)=W(A)(g^{-1}\cdot x)$;

\item Traciality: we have
\begin{equation}\int_M\,W(A)(x)W(B)(x)\,d\mu (x)=\Tr(AB). \nonumber
\end{equation}

\end{enumerate} \end{definition}

Let us consider the case when $G$ is a quasi-Hermitian Lie group and $\pi$ is a unitary representation of $G$ which is realized in a reproducing kernel Hilbert space $\mathcal H$ consisting of holomorphic functions on a complex domain \cite[Chapter XII]{Ne}. In this case,
the Berezin correspondence, introduced by F. A. Berezin in the 1970's in order 
to develop a program of quantization by deformation for complex domains \cite{Be1, Be2}, is covariant with respect to $\pi$. Moreover, the Berezin correspondence
is an isomorphism
from the Hilbert space of all Hilbert-Schmidt operators on $\mathcal
H$ (equipped with the Hilbert-Schmidt norm) onto a space of
square-integrable functions on a complex domain and the
isometric part $W$ in the polar decomposition of $S$ is a Stratonovich-Weyl correspondence,
see \cite{CaPad}. 

It should be noticed that, in general, one can't give an explicit formula for $W$ which allows
the computation of $W(A)$ for certain operators $A$ on $\mathcal H$. However, in some cases
of interest,  $\mathcal H$ is the Fock space and $W$ reduces to the complex Weyl correspondence which can be defined by an integral formula \cite{CaTo, CaTs}. This occurs in particular for the unitary representations of the Heisenberg group \cite{CaTo}, of the Heisenberg motion groups \cite{CaTo}, of the diamond group \cite{CaTs} and for the metaplectic representation \cite{CaComp, CaExt}.

The diamond group is the semi-direct product of the Heisenberg group by the real line.
The diamond group is one of the simplest examples of solvable non-exponential Lie groups, so 
it is used to test different methods and conjectures as, for instance, the construction of unitary representations from polarizations on coadjoint orbits \cite{Bern, Kir, Strea}, the continuity
of the Kirillov map \cite{Lud} and the separation of unitary representations by means of the moment map \cite{Abd}.

In \cite{CaTs} we proved that the complex Weyl correspondence is a Stratonovich-Weyl correspondence for the generic representations of the diamond group and we give closed formulas for the complex Weyl symbols of the representation operators (see also \cite{CaRiv}).  

The main goal of the present paper is to extend the results of \cite{CaRiv,CaTs} to the case of the generalized diamond group, which is more delicate. Let us briefly detail
below the content of the paper. 

We first review some generalities about the generic representations of the Heisenberg group $H_n$ on the Fock space, the Berezin correspondence and the complex Weyl correspondence (Sections \ref{sec:2} and \ref{sec:3}).

The generic representations of the generalized diamond group $G$ are then constructed by extending the generic representations of $H_n$ to $G$. This is done by solving some functional equation involving the kernels of the representation operators in the spirit of \cite{CaComp}
(Section \ref{sec:4}).

We establish that the Berezin correspondence and the complex Weyl correspondence
are covariant with respect to a generic representation of $G$ (Section \ref{sec:5}).

We compute the Berezin symbol and the complex Weyl symbol of the representation operators
$\pi(g)$ for $g\in G$ and $d\pi(X)$ for $X$ in the Lie algebra of $G$ 
(Section \ref{sec:6}).

We use $W$ to connect $\pi$ to a coadjoint orbit of $G$ in the spirit of the Kirillov-Kostant 
of orbits; this allows us to interpret $W$ in terms of a Stratonovich-Weyl correspondence
(Section \ref{sec:7}).

We also develop a Schr\"odinger model $\pi'$ for $\pi$ by means of the Bargmann transform
(Section \ref{sec:8}). We then obtain a Melher-type formula \cite{CR1}. Moreover, we show that the classical Weyl correspondence gives a Stratonovich-Weyl correspondence
for $\pi'$.

Finally, we give some applications of the preceding results to computations of star products;
for instance, we recover a classical formula for the Moyal product of two Gaussians
and we compute the star exponential (for the Moyal product) of some polynomials (Section \ref{sec:9}).

\section{Generic representations of the Heisenberg group on the Fock space} \label{sec:2}
This section and the next section are mostly of expository nature.
First, we review some facts on the Bargmann-Fock
model for the generic 
representations (that is, the unitary irreducible non degenerate representations) of the Heisenberg group, the Berezin correspondence and the Weyl correspondence on the Fock space.
We follow closely \cite{CaComp}, see also \cite{CaTs} and \cite{CaExt}. Our main references
for the Heisenberg group and its unitary irreducible representations are \cite{Fo, Kir, Tay};
for the Berezin calculus, \cite{Be1, Be2} and, for the Weyl correspondence and the 
Stratonovich-Weyl quantizer, \cite{AU, CR1, Gos, Fo,  GB, Ho1, Ner}.

For each $z,\,w \in {\mathbb C}^{n}$, let $zw:=\sum_{k=1}^nz_kw_k$. For each $z, z',w,w'\in {\mathbb C}^{n}$, let
\begin{equation*}\omega ((z,w),(z',w'))=\tfrac{i}{2}(zw'-z'w). \end{equation*}

Then the $(2n+1)$-dimensional real Heisenberg group is 
\begin{equation*}H_n:=\{(z,c)\,:\,z\in {\mathbb C}^n, c\in {\mathbb R}\}\end{equation*} 
equipped with the multiplication law
\begin{equation*}(z,c)\cdot (z',c')=(z+z',c+c'+\tfrac{1}{2}\omega ((z,{\bar z}),(z',{\bar
z'}))).\end{equation*}

Let $\lambda>0$. By the Stone-von Neumann
theorem, there exists a unique (up to unitary equivalence) unitary
irreducible representation $\rho_{\lambda}$ of $H_n$ whose restriction to the center
of $H_n$ is the character $(0,c)\rightarrow e^{i\lambda c}$
\cite{Kir, Tay}. 
The Bargmann-Fock realization of $\rho_{\lambda}$ is defined as follows \cite{Barg}. 

Let ${\mathcal F}_{\lambda}$ be the Hilbert space of all holomorphic functions $f$
on ${\mathbb C}^n$ such that \begin{equation*}\Vert f\Vert^2_{\lambda}
:=\int_{{\mathbb C}^n} \vert f(z)\vert^2\, e^{-\lambda \vert
z\vert^2/2}\,d\mu_{\lambda} (z) <+\infty\end{equation*} where
 $d\mu_{\lambda}(z):=(2\pi
)^{-n}{\lambda}^n\,dm(z)$. Here $z=x+iy$ with $x$ and $y$ in ${\mathbb
R}^n$ and $dm(z):= dx\,dy$ denotes the standard Lebesgue measure on ${\mathbb
C}^n$.

Then we have
\begin{equation*}({\rho}_{\lambda}(h)f)(z)=\exp \left(i\lambda c_0+\tfrac{\lambda}{2}{\bar z_0}z-\tfrac{\lambda}{4}\vert z_0\vert^2\right)\,f(z- z_0) \end{equation*}
for each $h=(z_0, c_0)\in H_n$ and $z\in {\mathbb C}^n$.

For each $z\in {\mathbb C}^n$, consider the \textit {coherent state}
$e_z(w)=\exp (\lambda{\bar z}w/2)$. Then we have the reproducing property
$f(z)=\langle f,e_z\rangle_{{\mathcal F}_{\lambda}}$ for each $f\in {\mathcal
F}_{\lambda}$.

Let us introduce the Berezin
calculus on ${\mathcal F}_{\lambda}$ \cite{Be1, Be2, CaRiv}.
The Berezin (covariant) symbol of an operator $A$ on ${\mathcal
F}_{\lambda}$ is the
function $S_{\lambda}(A)$ defined on ${\mathbb C}^n$ by
\begin{equation*}S_{\lambda}(A)(z):=\frac {\langle A\,e_z\,,\,e_z\rangle_{{\mathcal F}_{\lambda}}}{\langle e_z\,,\,e_z\rangle_{{\mathcal F}_{\lambda}}} \end{equation*}
and the double Berezin symbol $s_{\lambda}$ is defined by
\begin{equation*}s_{\lambda}(A)(z,w):=\frac {\langle A\,e_w\,,\,e_z\rangle_{{\mathcal F}_{\lambda}}}{\langle e_w\,,\,e_z\rangle_{{\mathcal F}_{\lambda}}} \end{equation*}
for each $(z,w)\in {\mathbb C}^n \times {\mathbb C}^n$ such that 
$\langle e_w\,,\,e_z\rangle_{{\mathcal F}_{\lambda}}\not= 0$.

Since $s_{\lambda}(A)(z,w)$ is holomorphic in the variable $z$ and anti-holomorphic in the variable $w$, $s_{\lambda}(A)$ is determined by its restriction to the diagonal of ${\mathbb C}^n \times {\mathbb C}^n$, that is, by $S_{\lambda}(A)$. Moreover, the operator $A$ can be recovered from $s_{\lambda}(A)$ as follows. We have

\begin{align*}
A\,f(z)&=\langle A\,f\,,\,e_z \rangle_{{\mathcal F}_{\lambda}}  
  =\langle f\,,\,A^{\ast}\,e_z\rangle_{{\mathcal F}_{\lambda}}  \\
&=\int _{{\mathbb C}^n}\,f(w)\overline {A^{\ast}\,e_z(w)}\,e^{-\lambda \vert
w\vert^2/2}\,d\mu_{\lambda} (w) \\
&=\int _{{\mathbb C}^n}\,f(w)\overline {\langle A^{\ast}\,e_z,e_w\rangle_{{\mathcal F}_{\lambda}}}\,e^{-\lambda \vert
w\vert^2/2}\,d\mu_{\lambda} (w) \\
&=\int _{{\mathbb C}^n} \,f(w)\,s_{\lambda}(A)(z,w)\langle e_w,e_z\rangle_{{\mathcal F}_{\lambda}}\,e^{-\lambda \vert
w\vert^2/2}\,d\mu_{\lambda}(w).\\ 
\end{align*}
This shows that  the map $A\rightarrow S_{\lambda}(A)$ is injective and that the
kernel of $A$ is
\begin{equation} k_A(z,w)=\langle Ae_w,e_z\rangle_{{\mathcal F}_{\lambda}}=s_{\lambda}(A)(z,w)\langle e_w,e_z\rangle_{{\mathcal F}_{\lambda}}.\end{equation}

The map $S_{\lambda}$ is a bounded operator from the
space ${\mathcal L}_2({\mathcal F}_{\lambda})$ of all Hilbert-Schmidt operators
on ${\mathcal F}_{\lambda}$ (endowed with the
Hilbert-Schmidt norm) to $L^2({\mathbb C}^n,\mu_{\lambda})$  which
is one-to-one and has dense range \cite{UU}. 

Now, we introduce 
the complex Weyl correspondence starting from a \textit{Stratonovich-Weyl quantizer}
see \cite{CaTo, GB, St} and \cite[Example 2.2 and Example 4.2]{AU}.

Let $R_0$ be the parity operator on ${\mathcal F}_{\lambda}$ defined by
\begin{equation*}(R_0f)(z)=2^nf(-z). \end{equation*}

Then we define the Stratonovich-Weyl quantizer $\Omega_0$ by
\begin{equation*}\Omega_0(z):=\rho_{\lambda}(z,0)R_0\rho_{\lambda}(z,0)^{-1} \end{equation*} for each $z\in {\mathbb C}^n$. 
Thus we get
\begin{equation}\label{eq:om}(\Omega_0(z)f)(w)=2^n\exp \left({\lambda}(w{\bar z}-\vert z\vert^2)\right) f(2z-w) \end{equation}
for each $z, w\in {\mathbb C}^n$ and $f\in {\mathcal F}_{\lambda}$. 

For each trace class operator $A$ on ${\mathcal F}_{\lambda}$, we define 
\begin{equation*} W_0(A)(z):=\Tr(A\Omega_0(z))\end{equation*}
for each $z \in {\mathbb C}^n$.
We have the following proposition, see \cite{AU, CaTo, CaTs}.

\begin{proposition} \label{prop:intW} For each trace class operator $A$ on ${\mathcal F}_{\lambda}$
and each $z \in {\mathbb C}^n$, we have
\begin{equation} \label{eq:intWsym}
W_0(A)(z)=2^n\int_{{\mathbb C}^n}k_A(z+w,z-w)\exp \left(\tfrac{\lambda}{2}\left(-z{\bar z}-w{\bar w}+z{\bar w}-{\bar z}w\right)\right) d\mu_{\lambda}(w).\end{equation}
\end{proposition}

This integral formula allows to extend $W_0$ to operators
on ${\mathcal F}_{\lambda}$  which are not necessarily trace class, for instance Hilbert-Schmidt operators. In particular, it is known that $W_0:{\mathcal L}_2({\mathcal F}_{\lambda})\rightarrow L^2({\mathbb C}^n,\mu_{\lambda})$ is the unitary part in the polar decomposition of $S_{\lambda}$ \cite{ CaRiv, CaTs}.

On the other hand, we can also consider the case of
the differential operators on ${\mathcal F}_{\lambda}$ with polynomial coefficients.

Here we use the standard multi-index notation. If $p=(p_1,p_2,\ldots,p_n)\in {\mathbb N}^n$, we set $z^{p}=z_1^{p_1}z_2^{p_2}\dots z_n^{p_n}$, $\vert p\vert=p_1+
p_2+\dots+p_n$, $p!=p_1!p_2!\dots p_n!$. Also, we say that $p\leq q$
if $p_k\leq q_k$ for each $k=1,2,\ldots, n$ and, in this case, we denote
$\tbinom{q}{p}=\tfrac{q!}{p!(q-p)!}$.

\begin{proposition}\label{prop:Wdiff} \rm{ \cite{CaTo}}
For each $p,q \in {\mathbb N}^n$, let $A_{pq}:=z^p(\frac{\partial}{\partial z})^q$. Then the integral in Equation \ref{eq:intWsym} is convergent and we have 
\begin{equation*}W_0(A_{pq})(z)=2^{-\vert q\vert} \sum_{k\leq p,q}(-1)^{\vert k \vert}\frac{p!\,q!}{k!(p-k)!(q-k)!}\lambda^{\vert q \vert-\vert k\vert}z^{p-k}{\bar z}^{q-k}.\end{equation*}
\end{proposition}

\section{The Schr\"odinger model for the generic representations of $H_n$} \label{sec:3}
Here, in order to connect $W_0$ to
the classical Weyl correspondence, we consider
another realization of the unitary irreducible representation of $H_n$ with central character
$(0,c)\rightarrow e^{i\lambda c}$, namely the Schr\"odinger representation $\rho'_{\lambda}$
defined on $L^2({\mathbb R}^n)$ by
\begin{equation*}(\rho'_{\lambda}(a+ib,c)\phi)(x)
=\exp \left(i\lambda (c-bx+\tfrac{1}{2}ab)\right)\,\phi(x-a) \end{equation*}
for each $a, b, x \in {\mathbb R}^n$.

In the setting of the method of orbits \cite{Kir}, $\rho'_{\lambda}$ can be obtained by using a real polarization of a coadjoint orbit of $H_n$ whereas $\rho_{\lambda}$ is obtained from a complex polarization of the same coadjoint orbit \cite{Bern, Kir}.

An (unitary) intertwining operator between $\rho_{\lambda}$ and $\rho'_{\lambda}$ is the Bargmann
transform $\mathcal B: L^2({\mathbb R}^n)\rightarrow {\mathcal F}_{\lambda}$  defined by
\begin{equation*}({\mathcal B}f)(z)=\left(\tfrac{\lambda}{\pi}\right)^{n/4}\,\int_{{\mathbb R}^n}\,\exp
\left(-\tfrac{\lambda}{4}z^2+\lambda zx-\tfrac{\lambda}{2}x^2\right)
\,\phi(x)\,dx,\end{equation*} 
see \cite{CaRiv, Fo}.

We can imitate the construction of $W_0$ from $\Omega_0$ given in 
Section \ref{sec:2}.
Let $R_1$ be the parity operator on $ L^2({\mathbb R}^n)$ defined by
\begin{equation*}(R_1\phi)(x)=2^n\phi(-x), \end{equation*}
Consider the Stratonovich-Weyl quantizer $\Omega_1$ on ${\mathbb R}^{2n}$ defined by
\begin{equation*}\Omega_1(a,b):=\rho'_{\lambda}(a+ib,0)R_1\rho'_{\lambda}(a+ib,0)^{-1}. \end{equation*} 
By an elementary computation, we get
\begin{equation}(\Omega_1(a,b)\phi)(x)=2^n\exp \left(2i{\lambda}b(a-x)\right) \phi(2a-x) \end{equation}
for each $\phi \in L^2({\mathbb R}^n)$.

For each trace class operator $A$ on $L^2({\mathbb R}^n)$, we define the function
$W_1(A)$ on ${\mathbb R}^{2n}$ by 
\begin{equation*} W_1(A)(x,y):=\Tr(A\Omega_1(x,y))\end{equation*}
for each $x, y \in {\mathbb R}^n$.

Observing that since $\mathcal B$ also intertwines $R_0$ and $R_1$ (that is, we have $\mathcal B R_1=R_0\mathcal B$), we can easily verify that, for each trace class operator $A$ on $L^2({\mathbb R}^{n})$ and each
$a, b\in {\mathbb R}^{n}$, we have
\begin{equation*}W_1(A)(a,b)=W_0({\mathcal B}A{\mathcal B}^{-1})(a+ib).\end{equation*}
This relation can be extended to operators which are not
necessarily of trace class.

Now, let us indicate the connection between $W_1$ and the classical Weyl correspondence 
${\mathcal W}$ on ${\mathbb R}^{2n}$ which can be defined as follows, see \cite{Fo, Ho1}.
For each function $f$ in the Schwartz space ${\mathcal S}({\mathbb
R}^{2n})$, we define the operator $ {\mathcal W}(f)$ acting on the Hilbert
space $L^2({\mathbb R}^{n})$ by \begin{equation}\label{eqcalWeyl} ({\mathcal W}(f)\phi)
(x)={(2\pi)}^{-n}\,\int_{{\mathbb R}^{2n}}\, e^{i yt} f( x+\tfrac{1}{2}y,
t)\,\phi (x+y)\,dy\,dt.
\end{equation}

In \cite{CaJLT, CaComp}, we proved that, for each $f\in {\mathcal S}({\mathbb R}^{2n})$,
we have 
\begin{equation*}W_1({\mathcal W}(f))(x,y)=f(x,\lambda y),
\end{equation*}
for each $x,y \in {\mathbb R}^n$.

\section{Generic representations of the generalized diamond group} \label{sec:4}
The generic representations of the diamond group can be obtained as holomorphically induced
representations by using the method of orbits \cite{Bern, Kir, Lud, Strea} or by using the general method of \cite{Ne}, see \cite{CaPad}. Here, we will construct the
generic representations of the generalized diamond group by extending those of $H_n$.
This rather elementary method is inspired by considerations on the metaplectic representation, see \cite{CaComp}.

Let $m$ be a positive integer. Let $\alpha_1,\alpha_2,\ldots,\alpha_n$ be $n$ linear forms on ${\mathbb R}^m$. We consider the action of ${\mathbb R}^m$ on ${\mathbb C}^n$ defined by
\begin{equation*}t\cdot z=t\cdot (z_1,z_2,\ldots,z_n):=(e^{i\alpha_1(t)}z_1,
e^{i\alpha_2(t)}z_2,\ldots,e^{i\alpha_n(t)}z_n). \end{equation*}

For $t\in {\mathbb R}^m$, we denote in general $t^{-1}$ instead of $-t$. Indeed, the notation $t^{-1}\cdot z$ seems to be more relevant than the notation
$(-t)\cdot z$.

The generalized diamond group is ${\mathbb R}^m\times {\mathbb C}^n \times 
{\mathbb R}$ with the multiplication
\begin{equation*}(t,z,c)\cdot (t',z',c')=(t+t', t'\cdot z+z',c+c'+\tfrac{1}{2}\omega ((z, {\bar z}),(t\cdot z', \overline{t\cdot z'})).\end{equation*}

Note that $H_n$ can be identified with the subgroup of $G$ consisting of the elements
of the form $(0,z,c)$ with $z\in {\mathbb C}^n$ and $c\in {\mathbb R}$.

Note also that the action of ${\mathbb R}^m$ on ${\mathbb C}^n$ gives an action of 
${\mathbb R}^m$ on $H_n$ defined by
\begin{equation*} t\cdot (z,c):=(t\cdot z,c), \quad t\in {\mathbb R}^m, z\in {\mathbb C}^n,
c\in {\mathbb R}.\end{equation*}
Then we see that $G$ is the semi-direct product ${\mathbb R}^m \rtimes H_n$ with respect to this action.

Now we construct the generic representations of $G$ from those of $H_n$.

\begin{proposition} \label{proppi} Let $\lambda>0$. For each $t\in {\mathbb R}^m$, let $\sigma(t)$ be an operator on ${\mathcal F}_{\lambda}$. Then the equation
\begin{equation*}\pi(t,h)=\rho_{\lambda}(h)\sigma(t)\quad t\in {\mathbb R}^m, \,h\in H_n \end{equation*} defined a unitary representation of $G$ on ${\mathcal F}_{\lambda}$ if and only if there exists a unitary character $\chi$ on ${\mathbb R}^m$ such that
\begin{equation*}(\sigma(t)f)(z)=\chi(t)f(t^{-1}\cdot z)\end{equation*} for each
$t\in {\mathbb R}^m$, $f\in {\mathcal F}_{\lambda}$ and $z\in {\mathbb C}^n$.

In this case, we have
\begin{equation*}({\pi}(t,z_0,c_0)f)(z)=\chi(t)\exp \left(i\lambda c_0+\tfrac{\lambda}{2}{\bar z_0}z-\tfrac{\lambda}{4}\vert z_0\vert^2\right)\,f(t^{-1}\cdot (z- z_0)) \end{equation*}
for each $t\in {\mathbb R}^m$,  $z, z_0\in {\mathbb C}^n$ and $c_0\in {\mathbb R}$.
\end{proposition}

\begin{proof} Assume that $\pi$ defined as above is a unitary representation of $G$ on ${\mathcal F}_{\lambda}$. Then we can write
\begin{equation*}\pi(t,h)\pi(t',h')=\pi((t,h)\cdot (t',h')),\quad t,t'\in  {\mathbb R}^m, \, h,h'\in H_n. \end{equation*} Thus we get
\begin{equation} \label{eq:comm} \rho_{\lambda}(t\cdot h)\sigma(t)=\sigma(t)\rho_{\lambda}( h), \quad 
t\in {\mathbb R}^m, \, h\in H_n. \end{equation}

Let us denote by $b_t(z,w)$ the kernel of $\sigma(t)$ for each $t\in {\mathbb R}^m$, that is, we have
\begin{equation}\label{eq:bt} (\sigma(t)f)(z)=\int_{{\mathbb C}^n}\,b_t(z,w)f(w)e^{-\lambda \vert w\vert^2/2}\,d\mu_{\lambda}(w)\end{equation}
for each $t\in {\mathbb R}^m$, $f\in {\mathcal F}_{\lambda}$ and $z\in {\mathbb C}^n$.

Let $h=(z_0,c_0)\in H_n$. Then, on the one hand, we have
\begin{align*}(\rho_{\lambda}&(t\cdot h)\sigma(t)f)(z)=
\exp \left(i\lambda c_0+\tfrac{\lambda}{2}z(\overline{t\cdot z_0})-\tfrac{\lambda}{4}\vert z_0\vert^2\right)\\
&\times \int_{{\mathbb C}^n}\,b_t(z-t\cdot z_0,w)f(w)e^{-\lambda \vert w\vert^2/2}\,d\mu_{\lambda}(w)\end{align*}
for each $t\in {\mathbb R}^m$, $f\in {\mathcal F}_{\lambda}$ and $z\in {\mathbb C}^n$.

On the other hand, we have
\begin{align*}(\sigma(t)\rho_{\lambda}(h)f)(z)=&
\int_{{\mathbb C}^n}\,b_t(z,w)\exp \left(i\lambda c_0+\tfrac{\lambda}{2}{\bar z_0}w-\tfrac{\lambda}{4}\vert z_0\vert^2\right)f(w-z_0)e^{-\lambda \vert w\vert^2/2}\,d\mu_{\lambda}(w)\\
=&\int_{{\mathbb C}^n}\,b_t(z,w+z_0)
\exp \left(i\lambda c_0-\tfrac{\lambda}{2}{\bar w}z_0-\tfrac{\lambda}{4}\vert z_0\vert^2\right)f(w)\,e^{-\lambda \vert w\vert^2/2}\,d\mu_{\lambda}(w),
\end{align*} 
for each $t\in {\mathbb R}^m$, $f\in {\mathcal F}_{\lambda}$ and $z\in {\mathbb C}^n$,
by the change  $w\rightarrow w+z_0$.

Then we can express Equation \ref{eq:comm} in terms of kernels as
\begin{equation}\label{eq:comm2}\exp \left(\tfrac{\lambda}{2}(\overline{t\cdot z_0})z \right)\,
b_t(z-t\cdot z_0,w)=\exp \left(-\tfrac{\lambda}{2}{\bar w} z_0 \right)\,b_t(z,w+z_0)
\end{equation} for each $t\in {\mathbb R}^m$ and each $z, z_0, w\in {\mathbb C}^n$.
Taking $w=0$ and then making the change $z_0\rightarrow w$ in Equation \ref{eq:comm2}, we get
\begin{equation*}b_t(z,w)\exp \left(-\tfrac{\lambda}{2}(\overline{t\cdot w})z \right)=b_t(z-t\cdot w,0)\end{equation*}
for each $z,w\in {\mathbb C}^n$. In this equation, the left-hand side is anti-holomorphic in the variable $w$ whereas the right-hand side is holomorphic in $w$. Consequently, for each 
$t\in {\mathbb R}^m$, there exists $\chi(t)\in {\mathbb C}$ such that 
\begin{equation*}b_t(z,w)=\chi(t)\exp \left(\tfrac{\lambda}{2}(\overline{t\cdot w})z \right)
\end{equation*} for each $z,w\in {\mathbb C}^n$. Replacing in Equation \ref{eq:bt}, we get,
for each $f\in \mathcal F_{\lambda}$, $t\in {\mathbb R}^m$ and $z\in {\mathbb C}^n$,
\begin{align*}(\sigma(t)&f)(z)=\chi(t)\, \int_{{\mathbb C}^n}\exp \left(\tfrac{\lambda}{2}(\overline{t\cdot w})z \right)\,f(w)e^{-\lambda \vert w\vert^2/2}\,d\mu_{\lambda}(w)\\
&=\chi(t)\langle f, e_{t^{-1}\cdot z}\rangle_{\mathcal F_{\lambda}}=\chi(t)f(t^{-1}\cdot z).
\end{align*}
Moreover, by writing $\pi(t+t',0,0)=\pi(t,0,0)\pi(t',0,0)$, we see that $\sigma(t+t')=\sigma(t)\sigma(t')$ hence $\chi(t+t')=\chi(t)\chi(t')$ for $t,t'\in {\mathbb R}^m$. This proves that $\chi$ is a character of ${\mathbb R}^m$. Finally, since $\pi(t,0,0)$ is a unitary operator, we have that $\chi$ is unitary.
\end{proof}

\section{Covariance of $S_{\lambda}$ and $W_0$} \label{sec:5}
Here we establish that $S_{\lambda}$ and $W_0$ are $G$-covariant with respect to $\pi$.
Covariance of $S_{\lambda}$ will follow from some identities about the action of $G$ on the coherent states $e_z$, $z\in {\mathbb C}^n$. Covariance of $W_0$ will be proved here by using the integral formula for $W_0$, see Proposition \ref{prop:intW}.

\begin{lemma}\label{lem:cs} Let $g=(t,h)\in G$ with $h=(z_0,c_0)\in H_n$. Then we have
\begin{equation}\label{eq:piez} \pi(g)e_z=\chi(t)\, \exp \left(i\lambda c_0-\tfrac{\lambda}{2}(t^{-1}\cdot z_0){\bar z}-\tfrac{\lambda}{4}\vert z_0\vert^2\right)\,e_{t\cdot z+z_0}\end{equation}
for each $z\in {\mathbb C}^n$. In particular, we have
\begin{equation}\label{eq:sigez} \sigma(t)e_z=\chi(t) \,e_{t\cdot z},\quad z\in {\mathbb C}^n \end{equation}
and
\begin{equation}\label{eq:rhoez} \rho_{\lambda}(h)e_z= \exp \left(i\lambda c_0-\tfrac{\lambda}{2}z_0{\bar z}-\tfrac{\lambda}{4}\vert z_0\vert^2\right)\,e_{z+z_0}\end{equation}
for each $z\in {\mathbb C}^n$.
\end{lemma}

\begin{proof} Let $t\in {\mathbb R}^m$ and $z\in {\mathbb C}^n$. Then, for each $w\in
{\mathbb C}^n$, we have 
\begin{equation*}(\sigma(t)e_z)(w)=\chi(t)e_z(t^{-1}\cdot w)=\chi(t)
\exp \left(\tfrac{\lambda}{2}{\bar z}(t^{-1}\cdot w)\right)=\chi(t)e_{t\cdot z}(w). \end{equation*} 
This proves Equation \ref{eq:sigez}. Similarly, for each $w\in
{\mathbb C}^n$, we can write
\begin{align*}
(\rho_{\lambda}(h)e_z)(w)=& \exp \left(i\lambda c_0+\tfrac{\lambda}{2}{\bar z_0}w-\tfrac{\lambda}{4}\vert z_0\vert^2\right)\,e_z(w-z_0)\\
=& \exp \left(i\lambda c_0+\tfrac{\lambda}{2}{\bar z_0}w-\tfrac{\lambda}{4}\vert z_0\vert^2\right)\,\exp \left(\tfrac{\lambda}{2}{\bar z}(w-z_0)\right)\\
=&\exp \left(i\lambda c_0-\tfrac{\lambda}{2}z_0{\bar z}-\tfrac{\lambda}{4}\vert z_0\vert^2\right)\,e_{z+z_0}(w)\end{align*}
hence we get Equation \ref{eq:rhoez}. By combining Equation \ref{eq:rhoez} with
Equation \ref{eq:sigez}, we obtain Equation \ref{eq:piez}.
\end{proof}

This leads us to introduce the action of $G$ on ${\mathbb C}^n$ defined by 
\begin{equation*} (t,z_0,c_0)\cdot z=t\cdot z+z_0,\quad t\in {\mathbb R}^m, z\,, z_0\in {\mathbb C}^n, c_0\in {\mathbb R}.\end{equation*}

\begin{proposition}\label{prop:covS} Let $A$ be an operator on $\mathcal F_{\lambda}$.
For each $g\in G$ and $z\in {\mathbb C}^n$, we have
\begin{equation*} S_{\lambda}(\pi(g)^{-1}A\pi(g))(z)=S_{\lambda}(A)(g\cdot z).\end{equation*}
\end{proposition}

\begin{proof} Let $g=(t,h)\in G$ with $h=(z_0,c_0)\in H_n$. Let $z\in {\mathbb C}^n$. In order to simplify the notation, we set
\begin{equation*}\beta(t,h,z):=\chi(t)\, \exp \left(i\lambda c_0-\tfrac{\lambda}{2}(t^{-1}\cdot z_0){\bar z}-\tfrac{\lambda}{4}\vert z_0\vert^2\right).\end{equation*}
We can then write Lemma \ref{lem:cs} as
\begin{equation*}\pi(g)e_z=\beta(t,h,z)e_{g\cdot z}.\end{equation*}
Consequently, for each operator $A$ on  $\mathcal F_{\lambda}$, we have
\begin{align*}\langle \pi(g)^{-1}A\pi(g)e_z,e_z\rangle_{\mathcal F_{\lambda}}
=&\langle A\pi(g)e_z,\pi(g)e_z\rangle_{\mathcal F_{\lambda}}\\
=&\vert \beta(t,h,z)\vert^2\, \langle Ae_{g\cdot z}, e_{g\cdot z}\rangle_{\mathcal F_{\lambda}}. \end{align*}
In the case when $A$ is the identity operator, we get
\begin{equation*}\langle e_z,e_z\rangle_{\mathcal F_{\lambda}}=
\vert \beta(t,h,z)\vert^2\, \langle e_{g\cdot z}, e_{g\cdot z}\rangle_{\mathcal F_{\lambda}}.\end{equation*}
Hence 
\begin{align*} S_{\lambda}(\pi(g)^{-1}A&\pi(g))(z)=\frac{\langle \pi(g)^{-1}A\pi(g)e_z,e_z\rangle_{\mathcal F_{\lambda}}}{\langle e_z,e_z\rangle_{\mathcal F_{\lambda}}}=\frac{\langle A\pi(g)e_z,\pi(g)e_z\rangle_{\mathcal F_{\lambda}}}{\langle e_z,e_z\rangle_{\mathcal F_{\lambda}}}\\
&=\frac{\langle Ae_{g\cdot z}, e_{g\cdot z}\rangle_{\mathcal F_{\lambda}}}{\langle e_{g\cdot z}, e_{g\cdot z}\rangle_{\mathcal F_{\lambda}}}
=S_{\lambda}(A)(g\cdot z). \end{align*}
\end{proof}

\begin{proposition}\label{prop:covW0} Let $A$ be an operator on $\mathcal F_{\lambda}$.
For each $g\in G$ and $z\in {\mathbb C}^n$, we have
\begin{equation*} W_0(\pi(g)^{-1}A\pi(g))(z)=W_0(A)(g\cdot z).\end{equation*}
\end{proposition}

\begin{proof} Let $A$ be an operator on $\mathcal F_{\lambda}$. For $t\in {\mathbb R}^m$ 
we define $A':=\sigma(t)^{-1}A\sigma(t)$. Then, for each $f\in \mathcal F_{\lambda}$ and each $z\in {\mathbb C}^n$,  we have
\begin{align*}(A'f)(z)=&\chi(t)^{-1}(A\sigma(t)f)(t\cdot z)\\
=&\int_{{\mathbb C}^n}\,k_A(t\cdot z,w)\,f(t^{-1}\cdot w)e^{-\lambda \vert w\vert^2/2}\,d\mu_{\lambda}(w)\\
=&\int_{{\mathbb C}^n}\,k_A(t\cdot z,t\cdot w)\,f(w)e^{-\lambda \vert w\vert^2/2}\,d\mu_{\lambda}(w).
\end{align*}
Thus the kernel of $A'$ is $k_{A'}(z,w)=k_A(t\cdot z,t\cdot w)$. Hence we have
\begin{equation*}W_0(A')(z)=2^n\int_{{\mathbb C}^n}k_A(t\cdot (z+w),t\cdot (z-w))\exp \left(\tfrac{\lambda}{2}\left(-z{\bar z}-w{\bar w}+z{\bar w}-{\bar z}w\right)\right) d\mu_{\lambda}(w)\end{equation*} 
and, by making the change of variables $w\rightarrow t^{-1}\cdot w$,
we obtain $W_0(A')(z)=W_0(A)(t\cdot z)$. This proves the covariance property for $g$ of the form $(t,0,0)$. Similarly, for $h=(z_0,c_0)\in H_n$, let $A'':=\rho_{\lambda}(h)^{-1}A\rho_{\lambda}(h)$. For each $z,w \in {\mathbb C}^n$, we have
\begin{align*}
k_{A''}(z,w)&=\langle A''e_w,e_z\rangle_{\mathcal F_{\lambda}}\\
&=\langle A\rho_{\lambda}(h)e_w,\rho_{\lambda}(h)e_z\rangle_{\mathcal F_{\lambda}}\\
&= \exp \left(-\tfrac{\lambda}{2}z_0{\bar w}-\tfrac{\lambda}{2}\bar{z_0}z-\tfrac{\lambda}{2}\vert z_0\vert^2\right)
\langle Ae_{w+z_0},e_{z+z_0}\rangle_{\mathcal F_{\lambda}}\\
&= \exp \left(-\tfrac{\lambda}{2}z_0{\bar w}-\tfrac{\lambda}{2}\bar{z_0}z-\tfrac{\lambda}{2}\vert z_0\vert^2\right)k_A(z+z_0,w+z_0).
\end{align*} After some easy computations starting from Equation \ref{eq:intWsym}, this implies that 
\begin{equation*}W_0(A'')(z)=W_0(A)(z+z_0), \quad z \in {\mathbb C}^n.\end{equation*}
Hence we have covariance of $W_0$ for each $g=(0,h)\in G$ with $h\in H_n$. Since
each $g\in G$ can be written as the product of an element of the form $(t,0,0)$ by an element
of the form $(0,h)$ we have proved the desired result.
\end{proof}

\section{Complex Weyl symbols of representation operators} \label{sec:6}
In this section, we compute $S(\pi(g))$ and $W_0(\pi(g))$ for $g\in G$.

\begin{proposition} \label{propSpi} For each $g=(t,z_0,c_0)\in G$, the kernel of $\pi(g)$
is
\begin{equation*}k_{\pi(g)}(z,w)=\chi(t) e^{i\lambda c_0}
 \exp \left(\tfrac{\lambda}{2}{\bar z_0}z+\tfrac{\lambda}{2}{\bar w}
(t^{-1}\cdot (z-z_0))
-\tfrac{\lambda}{4}\vert z_0\vert^2\right). \end{equation*}
Consequently, we have
\begin{equation*}S(\pi(g))(z)=\chi(t) e^{i\lambda c_0}
 \exp \left(\tfrac{\lambda}{2}{\bar z_0}z+\tfrac{\lambda}{2}{\bar z}
(t^{-1}\cdot (z-z_0))-\tfrac{\lambda}{2}\vert z \vert^2
-\tfrac{\lambda}{4}\vert z_0\vert^2\right)\end{equation*}
for each $z\in {\mathbb C}^n$.
\end{proposition}

\begin{proof} Let  $g=(t,z_0,c_0)\in G$. By the reproducing property, we can write
\begin{equation*}k_{\pi(g)}(z,w)=\langle \pi(g)e_w,e_z\rangle_{\mathcal F_{\lambda}}=(\pi(g)e_w)(z)\end{equation*} for each $z,w\in {\mathbb C}^n$. Then we see that the result follows from Proposition \ref{proppi}.
\end{proof}

In order to compute $W_0(\pi(g))$ for $g\in G$ we need the following lemma about the computation of Gaussian integrals. This lemma is a variant of \cite[Theorem 3, p. 258]{Fo}.

For $z\in {\mathbb C}$, we define $z^{1/2}$ as the principal determination of the square root (with branch cut along the negative real axis).

\begin{lemma} \cite{CaComp} \label{lemgauss} Let $A,B, D$ be $n\times n$ complex matrices such that
$A^t=A, D^t=D$. Let $M=\bigl(\begin{smallmatrix} A&B^t\\
B&D \end{smallmatrix}\bigr)$, $U=\bigl(\begin{smallmatrix} I_n&iI_n\\
I_n&-iI_n \end{smallmatrix}\bigr)$ and $N=U^tMU$. Assume that $\Rea (N)$ is positive definite. Let $u,v \in {\mathbb C}^n$. Then we have 
\begin{align*}\int_{{\mathbb C}^n}&\exp\left(-\left( w(Aw)+{\bar w}(D{\bar w})+2{\bar w}(Bw)\right)\right)\exp (uw+v{\bar w})\,dm(w)\\
=&\pi^n(\Det N)^{-1/2}\exp \left( \tfrac{1}{4}\begin{pmatrix}u&v\end{pmatrix}M^{-1}\begin{pmatrix}u\\v\end{pmatrix}\right).
\end{align*}
\end{lemma}

For $t\in {\mathbb R}^m$, it is convenient to introduce the diagonal matrix
\begin{equation*}A(t):=\Diag(e^{i\alpha_1(t)},e^{i\alpha_2(t)},\ldots,e^{i\alpha_n(t)}).\end{equation*} Then we have $t\cdot z=A(t)z$ for each $t\in {\mathbb R}^m$ and 
$z \in {\mathbb C}^n$.

\begin{proposition} \label{propWpi} Let  $g=(t,z_0,c_0)\in G$ such that $\alpha_k(t)\notin
\pi+2\pi{\mathbb Z}$ for each $k=1,2,\ldots, n$. Then, for each $z \in {\mathbb C}^n$, we have
\begin{align*}W_0(\pi(g))(z)&=2^n\chi(t) e^{i\lambda c_0}\Det(I_n+A(t^{-1}))^{-1}
 \exp \left(-{\lambda}(t^{-1}\cdot z_0){\bar z}
-{\lambda}\vert z\vert^2
-\tfrac{\lambda}{4}\vert z_0\vert^2\right)\\
&\times \exp \left( \tfrac{\lambda}{2}(t^{-1}\cdot z_0+2z)(I_n+A(t))^{-1}(\overline{t^{-1}\cdot z_0+2z})\right),
\end{align*} 
and, equivalently,
\begin{align*}W_0&(\pi(g))(z)=2^n\chi(t) e^{i\lambda c_0}
 \exp \left(-{\lambda}(t^{-1}\cdot z_0){\bar z}
-{\lambda}\vert z\vert^2
-\tfrac{\lambda}{4}\vert z_0\vert^2\right)\\
&\times \prod_{k=1}^n(1+e^{-i\alpha_k(t)})^{-1}\exp \left( \tfrac{\lambda}{2}
\sum_{k=1}^n(1+e^{i\alpha_k(t)})^{-1}\vert e^{-i\alpha_k(t)}a_k+2z_k\vert^2\right)
\end{align*} 
where $z_0=(a_1,a_2,\ldots,a_n)$.
\end{proposition}

\begin{proof} Let $g=(t,z_0,c_0)\in G$. By performing the change of variables $w\rightarrow w-z_0$ in Equation \ref{eq:intWsym}, we get
\begin{equation} \label{eq:intWbis}
W_0(\pi(g))(z)=2^n\int_{{\mathbb C}^n}k_{\pi(g)}(w,2z-w)\exp \left({\lambda}\left(-z{\bar z}+z{\bar w}-\tfrac{1}{2}w{\bar w}\right)\right) d\mu_{\lambda}(w). 
\end{equation}
In this equation, we replace $k_{\pi(g)}(w,2z-w)$ by its expression derived from Proposition \ref{propSpi}. Then, introducing the notation
\begin{equation*}I(t,z_0,z):=\int_{{\mathbb C}^n}\exp \left(\tfrac{\lambda}{2}w({\bar z_0}+2\overline{t\cdot z})+\tfrac{\lambda}{2}{\bar w}(t^{-1}\cdot z_0+2z)-
\tfrac{\lambda}{2}({\bar w}(t^{-1}\cdot w)+w{\bar w})
\right) dm(w),\end{equation*}
we get
\begin{equation*}W_0(\pi(g))(z)=\left(\tfrac{\lambda}{\pi}\right)^n\chi(t) e^{i\lambda c_0}
 \exp \left(-{\lambda}{\bar z}(t^{-1}\cdot z_0)
-{\lambda}\vert z\vert^2
-\tfrac{\lambda}{4}\vert z_0\vert^2\right)I(t,z_0,z). \end{equation*}
The computation of $I(t,z_0,z)$ can be performed by using Lemma \ref{lemgauss}. With
the notation as in the lemma we take $A=D=0$, $B=\tfrac{\lambda}{4}(I_n+A(t^{-1}))$ and
\begin{equation*} u=\tfrac{\lambda}{2}({\bar z_0}+2\overline{t\cdot z});\quad
v=\tfrac{\lambda}{2}(t^{-1}\cdot z_0+2z).\end{equation*}
In this case, we have
\begin{equation*}N=\tfrac{\lambda}{2}\,\begin{pmatrix}I_n+A(t)^{-1}&0\\0&I_n+A(t)^{-1}\\\end{pmatrix}\end{equation*} hence
\begin{equation*}\Rea(N)=\tfrac{\lambda}{2}\Diag(1+\cos \alpha_1(t),\ldots,1+\cos \alpha_n(t),1+\cos \alpha_1(t),\ldots,1+\cos \alpha_n(t)).
\end{equation*}
Assuming that $\alpha_k(t)\notin
\pi+2\pi{\mathbb Z}$ for each $k=1,2,\ldots, n$, the matrix
$\Rea(N)$ is positive definite. Moreover, we have
\begin{equation*}\Det(N)=\left(\tfrac{\lambda}{2}\right)^{2n}
\Det(I_n+A(t)^{-1})^2\end{equation*} and
\begin{equation*}\tfrac{1}{4}\begin{pmatrix}u&v\end{pmatrix}M^{-1}\begin{pmatrix}
u\\v\end{pmatrix}=\tfrac{\lambda}{2}(t^{-1}\cdot z_0+2z)(I_n+A(t))^{-1}(\overline {t^{-1}\cdot z_0+2z}).
\end{equation*} The result follows.
\end{proof}

Let ${\mathfrak h}_n$ be the Lie algebra of $H_n$ and $\mathfrak g$ be the Lie algebra of $G$. We write the elements of $\mathfrak g$ as $(t,u,c)$ with $t\in {\mathbb R}^m$,
$u \in {\mathbb C}^n$ and $c\in {\mathbb R}$. The Lie brackets of $\mathfrak g$ are
\begin{equation*}[(t,u,c),(t',u',c')]=(0,i(\alpha(t)u'-\alpha(t')u),\omega((u,{\bar u}),(u',{\bar u'})))\end{equation*} with the notation 
\begin{equation*}\alpha(t)u=(\alpha_1(t)u_1,\alpha_2(t)u_2,\ldots,\alpha_n(t)u_n),\quad
t\in {\mathbb R}^m, u=(u_1,u_2,\ldots,u_n)\in {\mathbb C}^n.
\end{equation*}

From Proposition \ref{propSpi} and Proposition \ref{propWpi}, we easily deduce the following result by differentiation.

\begin{proposition} \label{propWdpi} Let $X=(t,u,c)\in \mathfrak g$. Then we have
for each $z\in {\mathbb C}^n$
\begin{equation*}S(d\pi(X))(z)=d\chi(t)+i\lambda c+\tfrac{\lambda}{2}({\bar u}z-
{\bar z}u)-\tfrac{\lambda}{2}i{\bar z}(\alpha(t)z)\end{equation*} and
\begin{equation*}W_0(d\pi(X))(z)=d\chi(t)+i\lambda c+\tfrac{\lambda}{2}({\bar u}z-
{\bar z}u)+\tfrac{1}{2}i\sum_{k=1}^n\alpha_k(t)(1-\lambda \vert z_k\vert^2).\end{equation*}
\end{proposition}

\section{Stratonovich-Weyl correspondence for $G$} \label{sec:7}
In this section, we use covariance of $W_0$ in order to connect $\pi$ with some coadjoint orbit
of $G$. Then we interpret $W_0$ as a Stratonovich-Weyl correspondence for $G$.

First, we introduce some additional notation. Let ${\mathfrak g}^{\star}$ be the dual of
${\mathfrak g}$. Let 
$s\in ( {{\mathbb R}^m})^{\ast}$ (the dual of ${{\mathbb R}^m}$), $v\in {\mathbb C}^n$ and $d\in {\mathbb R}$.
Then we denote by $\xi=(s,v,d)$ the element of ${\mathfrak g}^{\ast}$ defined as follows.
For each $X=(t,u,c)\in \mathfrak g$, we have
\begin{equation*}\langle \xi,X\rangle=\langle s,t\rangle+\omega((v,{\bar v}),(u,{\bar u}))
+cd. \end{equation*}

\begin{proposition}\label{proppsi} \begin{enumerate}
\item There exists a map $\psi:{\mathbb C}^n\rightarrow {\mathfrak g}^{\ast}$ such that
\begin{equation*}W_0(d\pi(X))(z)=i\langle \psi (z), X\rangle\end{equation*}
for each $X\in  {\mathfrak g}$ and each $z\in {\mathbb C}^n$. Then we have
\begin{equation*}\psi(g\cdot z)=\Ad^{\ast}(g)\,\psi(z)\end{equation*}
for each $g\in G$ and each $z\in {\mathbb C}^n$;
\item  For each $z\in {\mathbb C}^n$, we have
\begin{equation*}\psi (z)=\left(-id\chi+\tfrac{1}{2}\sum_{k=1}^n(1-\lambda \vert z_k\vert^2)\alpha_k, -\lambda z,\lambda\right).\end{equation*}
\item Moreover, $\psi$ is a bijection from ${\mathbb C}^n$ onto the orbit ${\mathcal O}(\xi_0
)$ of the element \begin{equation*}\xi_0:=\left(-id\chi+\tfrac{1}{2}\sum_{k=1}^n\alpha_k,0,\lambda \right)\in {\mathfrak g}^{\ast} \end{equation*} 
for the coadjoint action of $G$.
\end{enumerate}
\end{proposition}

\begin{proof} First, we remark that for each $z\in {\mathbb C}^n$ the linear form defined on ${\mathfrak g}$ by
$X\rightarrow -iW_0(d\pi(X))(z)$ is real-valued by Proposition \ref{propWdpi}, hence it defines an element $\psi( z)$ in ${\mathfrak g}^{\ast}$.

By the covariance of $W_0$ with respect to $\pi$, for each $g\in G$, $X\in  {\mathfrak g}$ and $z\in {\mathbb C}^n$, we have 
\begin{align*}\langle \psi_(g\cdot z), X\rangle &=-iW(d\pi(X))(g\cdot z)\\
&=-iW(\pi(g)^{-1}d\pi(X)\pi(g))(z)\\
&=-iW(d\pi(\Ad(g)^{-1}X))(z)\\
&=\langle \psi(z), \Ad(g)^{-1}X\rangle\\
&= \langle \Ad^{\ast}(g)\psi (z), X\rangle
\end{align*} hence $ \psi(g\cdot z)=\Ad^{\ast}(g)\psi(z)$.
The rest of the proposition follows easily from Proposition \ref{propWdpi}.
\end{proof}

As a consequence of Proposition \ref{proppsi}, we can interpret our results in the context of Definition \ref{def1}.

Let $\nu_{\lambda}$ denote the measure $\psi_{\ast}(\mu_{\lambda})$ on ${\mathcal O}(\xi_0)$. Recall that ${\mathcal L}_2({\mathcal F}_{\lambda})$ denotes the space of all
Hilbert-Schmidt operators on ${\mathcal F}_{\lambda}$. 

\begin{proposition}\label{propSW} \begin{enumerate}
\item The map $W_0:{\mathcal
L}_2({\mathcal F}_{\lambda})\rightarrow L^2({\mathbb C}^n,\mu_{\lambda})$
is a Stratonovich-Weyl correspondence for the triple $( G,\pi,
{\mathbb C}^n)$;
\item The map $W'_0:{\mathcal
L}_2({\mathcal F}_{\lambda})\rightarrow L^2({\mathcal O}(\xi_0),\nu_{\lambda})$
defined by $W'_0(A)=W_0(A)\circ \psi^{-1}$
is a Stratonovich-Weyl correspondence for the triple $( G,\pi,
{\mathcal O}(\xi_0))$.
\end{enumerate}
\end{proposition}

\begin{proof}  (1) is a consequence of the covariance of $W_0$ with respect to $\pi$ and of the unitarity of  $W_0:{\mathcal
L}_2({\mathcal F}_{\lambda})\rightarrow L^2({\mathbb C},\mu_{\lambda})$, see  \cite{CaTs}.
(2) can be deduced from (1).
\end{proof}

\section{Schr\"odinger model for $\pi$} \label{sec:8}

The Schr\"odinger model for $\pi$ is the representation $\pi'$ of $G$ on $L^2({\mathbb R}^n)$ which is obtained by translating $\pi$ by means of the Bargmann transform $\mathcal B$, that is, $\pi'$ is defined by $\pi'(g)={\mathcal B}^{-1} \pi(g){\mathcal B}$ for each $g\in G$.
Then, by writing $g=(t,h)$ with $t\in {\mathbb R}^m$ and $h\in H_n$, we have
\begin{equation*}\pi'(g)={\mathcal B}^{-1}\rho_{\lambda}(h)\sigma(t){\mathcal B}=({\mathcal B}^{-1}\rho_{\lambda}(h){\mathcal B})({\mathcal B}^{-1}\sigma(t){\mathcal B})
=\rho '_{\lambda}(h)({\mathcal B}^{-1}\sigma(t){\mathcal B}).\end{equation*}
This leads us to consider $\sigma'(t):={\mathcal B}^{-1}\sigma(t){\mathcal B}$ for $t\in {\mathbb R}^m$. The aim of this section is to give explicit formulas for the kernel of $\sigma'(t)$ hence for the kernel of $\pi'(g)$.

For convenience we write ${\mathcal B}$ as
\begin{equation*}({\mathcal B}\phi)(z)=\int_{{\mathbb R}^n}B(z,x)\phi(x)\,dx\end{equation*} where
\begin{equation*}B(z,x):=\left(\tfrac{\lambda}{\pi}\right)^{n/4} \exp
\left(-\tfrac{\lambda}{4}z^2+\lambda zx-\tfrac{\lambda}{2}x^2\right).
\end{equation*}

Since ${\mathcal B}$ is unitary \cite{Fo}, we have
\begin{equation*}
\langle {\mathcal B}\phi, f\rangle_{{\mathcal F}_{\lambda}}=\langle \phi, {\mathcal B}^{-1}f\rangle_{L^2({\mathbb R}^n)}, \quad \phi \in L^2({\mathbb R}^n), f\in {\mathcal F}_{\lambda}.\end{equation*} This gives
\begin{equation*}({\mathcal B}^{-1}f)(x)=\int_{{\mathbb C}^n}\overline{B(z,x)}f(z) e^{-\lambda \vert
z\vert^2/2}\,d\mu_{\lambda}(z).\end{equation*}
Let us denote by $b'_t$ the kernel of $\sigma'(t)$ for $t\in {\mathbb R}^m$, that is, we have
\begin{equation*} (\sigma'(t)\phi)(x)=\int_{{\mathbb R}^n}b'_t(x,y)\phi(y)\,dy.\end{equation*}

\begin{proposition}\label{propbt} For each $t\in {\mathbb R}^m$ such that $\alpha_k(t)\notin \pi{\mathbb Z}$ for each $k=1,2,\ldots, n$, we have
\begin{align*}&b_t(x,y)=\left(\tfrac{\lambda}{\pi}\right)^{n/2}\chi(t) \Det(I_n-A(t^{-1})^2)^{-1/2}\exp \left(\tfrac{\lambda}{2}(x^2+y^2)\right )\\
&\times
\exp \left( \lambda (y(A(t^{-1})^2-I_n)^{-1}y-2xA(t^{-1})(A(t^{-1})^2-I_n)^{-1}y+x(A(t^{-1})^2-I_n)^{-1}x\right),
\end{align*}
or, equivalently,
\begin{align*}&b_t(x,y)=\left(\tfrac{\lambda}{\pi}\right)^{n/2}\chi(t)\prod_{k=1}^n(1-e^{-2i\alpha_k(t)})^{-1/2}\\ 
&\times
\exp \left(\tfrac{\lambda}{2}i\sum_{k=1}^n(\tan(\alpha_k(t)))^{-1}(x_k^2+y_k^2)
-\lambda i \sum_{k=1}^n(\sin(\alpha_k(t)))^{-1}x_ky_k \right).
\end{align*}
\end{proposition}

\begin{proof} Let $\phi \in {\mathbb R}^n$.  Then we have
\begin{equation*} (\sigma(t){\mathcal B}\phi)(z)=\chi(t)({\mathcal B}\phi)(t^{-1}\cdot z)
=\chi(t)\,\int_{{\mathbb R}^n}B(t^{-1}\cdot z,y)\phi(y)\,dy\end{equation*}
hence
\begin{align*}
({\mathcal B}^{-1}\sigma(t){\mathcal B}&\phi)(x)=
\int_{{\mathbb C}^n}\overline{B(z,x)}(\sigma(t){\mathcal B}\phi)(z) e^{-\lambda \vert
z\vert^2/2}\,d\mu_{\lambda}(z)\\
&=\chi(t)\, \int_{{\mathbb R}^n}\int_{{\mathbb C}^n}\overline{B(z,x)}B(t^{-1}\cdot z,y)\phi(y) e^{-\lambda \vert
z\vert^2/2}dy\,d\mu_{\lambda}(z).
\end{align*} 
This gives
\begin{equation*}b_t(x,y)=\chi(t)\, \int_{{\mathbb C}^n}\overline{B(z,x)}B(t^{-1}\cdot z,y)\phi(y) e^{-\lambda \vert
z\vert^2/2}\,d\mu_{\lambda}(z).\end{equation*}
Equivalently, introducing the integral
\begin{equation*}I_t(x,y):=\int_{{\mathbb C}^n}\exp \left( -\tfrac{\lambda}{4}((t^{-1}\cdot z)^2+{\bar z}^2)-\tfrac{\lambda}{2}\vert z\vert^2+\lambda ({\bar z}x+(t^{-1}\cdot z)y)\right)
dm(z),\end{equation*} we can write
\begin{equation*}b_t(x,y)=\left(\tfrac{\lambda}{\pi}\right)^{n/2}\left(\tfrac{\lambda}{2\pi}\right)^{n}\chi(t) \exp \left(-\tfrac{\lambda}{2}(x^2+y^2)\right)\,I_t(x,y).
\end{equation*}
The rest of the proof consists in computing $I_t(x,y)$ by using Lemma \ref{lemgauss}.
With the notation as in the lemma, we take
\begin{equation*}M=\begin{pmatrix} A&B\\B^t&D\end{pmatrix}=\tfrac{\lambda}{4}
\begin{pmatrix} A(t^{-1})^2&I_n\\I_n&I_n\end{pmatrix}
\end{equation*} and $u=\lambda t^{-1}\cdot y, v=\lambda x$. Indeed, with this choice,
we have
\begin{equation*}z(Az)=\tfrac{\lambda}{4}( t^{-1}\cdot z)^2;\quad {\bar z}(D{\bar z})=
\tfrac{\lambda}{4}{\bar z}^2;\quad 2{\bar z}(Bz)=\tfrac{\lambda}{2}z{\bar z}.
\end{equation*} 
We have to verify that the matrice $\Rea(N)$ which is here equal to
\begin{equation*}\tfrac{\lambda}{4}\begin{pmatrix}3I_n+ \Diag( \cos(2\alpha_1(t)),\ldots,\cos(2\alpha_n(t)))&\Diag(\sin(2\alpha_1(t)),\ldots, \sin(2\alpha_n(t)))\\\Diag(\sin(2\alpha_1(t)),\ldots, \sin(2\alpha_n(t)))&I_n-\Diag( \cos(2\alpha_1(t)),\ldots,
\cos(2\alpha_n(t)))\end{pmatrix}
\end{equation*} is positive define. Since the associated quadratic form is
\begin{align*}\begin{pmatrix}x&y\end{pmatrix}&\Rea(N)\begin{pmatrix}
x\\y\end{pmatrix}\\=& \tfrac{\lambda}{4}\sum_{k=1}^n\left((3+\cos(2\alpha_k(t)))x_k^2+
2\sin(2\alpha_k(t))x_ky_k+(1-\cos(2\alpha_k(t)))y_k^2\right),
\end{align*} it is sufficient to consider the case $n=1$. But for each $k=1,2,\ldots,n$, the matrix
\begin{equation*} \begin{pmatrix}3+\cos(2\alpha_k(t))&\sin(2\alpha_k(t))\\
\sin(2\alpha_k(t))&1-\cos(2\alpha_k(t))\end{pmatrix}
\end{equation*} is clearly positive definite because firstly we have $3+\cos(2\alpha_k(t))>0$
and secondly it has determinant $2(1-\cos(2\alpha_k(t)))$ which is positive under the hypothesis that $\alpha_k(t)\notin \pi{\mathbb Z}$ for each $k=1,2,\ldots,n$.

Moreover, we have that 
\begin{equation*} M^{-1}=\tfrac{4}{\lambda}\begin{pmatrix}(A(t^{-1})^2-I_n)^{-1}&-(A(t^{-1})^2-I_n)^{-1}\\-(A(t^{-1})^2-I_n)^{-1}&I_n+(A(t^{-1})^2-I_n)^{-1}
\end{pmatrix},
\end{equation*} then
\begin{align*}\tfrac{1}{4}&\begin{pmatrix}u&v\end{pmatrix}M^{-1}\begin{pmatrix}u\\v\end{pmatrix}\\
&=\lambda \left(x^2+x(A(t^{-1})^2-I_n)^{-1}x+y^2+y(A(t^{-1})^2-I_n)^{-1}y-2xA(t^{-1})
(A(t^{-1})^2-I_n)^{-1}y\right).\end{align*}
On the other hand, we also have
\begin{equation*}\Det(N)=\left(\tfrac{\lambda}{2}\right)^{2n}\Det(I_n-A(t^{-1})^2).
\end{equation*}
We are then in position to apply Lemma \ref{lemgauss}. The result hence follows.
\end{proof}

Let $A$ be an operator on ${\mathcal F}_{\lambda}$ and let $A'={\mathcal B}^{-1}A
{\mathcal B}$. Denote the kernel of $A$ ky $k_A(z,w)$ and the kernel of $A'$ by $K_{A'}(x,y)$.
Then the holomorphic function $k_A(z,{\bar w})$ is the $2n$-dimensional Bargmann transform
of $K_A$ \cite[Proposition 1.81]{Fo}. This fact can be used to provide another proof of Proposition \ref{propbt} which is more complicated than the proof given above. On the other hand, we easily deduce from Proposition \ref{propbt} the following result.

\begin{corollary}\label{corbt} Let $g=(t, a+ib,c)\in G$. Then the kernel of $\pi'(g)$ is
\begin{equation*} K_{\pi'(g)}(x,y)=\exp \left(i\lambda (c-bx+\tfrac{1}{2}ab)\right)\,b_t(x-a,y).\end{equation*}
\end{corollary}

The formulas given in Proposition \ref{propbt} and Corollary \ref{corbt} are analogous to the so-called Mehler formula, see \cite{CR1, Fo} and for recent developments, see \cite{PS}.

Let $j:{\mathbb R}^{2n}\rightarrow {\mathbb C}^n$ defined by $j(x,y)=x+iy$.
Let $\mu_{\lambda}'$ be the Lebesgue measure on ${\mathbb R}^{2n}$ normalized as
$j_{\ast}(\mu_{\lambda}')=\mu_{\lambda}$. Then, from Proposition \ref{propSW},
we immediatly obtain the following result.
  
\begin{proposition}\label{propSWbis} \begin{enumerate}
\item The map $W_1:{\mathcal
L}_2(L^2({\mathbb R}^{n}))\rightarrow L^2({\mathbb R}^{2n},\mu_{\lambda}')$
is a Stratonovich-Weyl correspondence for the triple $( G,\pi',
{\mathbb R}^{2n})$;
\item The map $W'_1:{\mathcal
L}_2(L^2({\mathbb R}^{n}))\rightarrow L^2({\mathcal O}(\xi_0),\nu_{\lambda})$
defined by $W'_1(A)=W_1(A)\circ (\psi \circ j)^{-1}$
is a Stratonovich-Weyl correspondence for the triple $( G,\pi',
{\mathcal O}(\xi_0))$.
\end{enumerate}
\end{proposition}

\section{Applications to Star products} \label{sec:9}

\subsection{Generalities} \label{sub:91}

We begin by introducing two associative products with $W_0$ and $W_1$
and we compare them to the Moyal product \cite{GBliv}. The Moyal product can be introduced as follows. First, we recall that $\mathcal W$ can be extended to polynomials \cite{Ho1}. More precisely,
if $f(x,y)=p(x)y^{s}$ where $p$ is a polynomial on ${\mathbb R}^n$ then we have
\begin{equation*}\label{eq:Wpoly} ({\mathcal W}(f)\phi )(x)=\left(i\frac {\partial }{ \partial
y}\right)^{s} \left( p(x+\tfrac{1}{2}y)\,\varphi (x+y) \right) \Bigl
\vert _{y=0},\end{equation*}  see for instance \cite{Vo}. 
Consequently, if $f$ is a polynomial then ${\mathcal W}(f)$ is a differential operator with polynomial coefficients. We can verify that ${\mathcal W}$ induces a bijection between the space of all polynomials on ${\mathbb R}^{2n}$ and the space of all differential operators on   ${\mathbb R}^{n}$ with polynomial coefficients. The Moyal product $\ast_M$ is then defined by  \begin{equation}\label{eq:M}{\mathcal W}(f\ast_M g)={\mathcal W}(f){\mathcal W}(g)\end{equation} for each polynomials $f, g$ on ${\mathbb R}^{2n}$. 

It is also known that
we can obtain an expansion of $f\ast_M g$ as follows.
Let
$u=(x,y)\in  {\mathbb R}^{n}\times  {\mathbb R}^{n}$. Then we have $u_i=x_i$ for $1\leq i \leq n$ and $u_i=y_{i-n}$ for $n+1\leq i\leq 2n$. For $f,g$ polynomials on ${\mathbb R}^{2n}$, define $P^0(f,g):=fg$,
\begin{equation*}P^1(f,g):=\sum_{k=1}^n\left( \frac{\partial f}{\partial x_k}
\frac{\partial g}{\partial y_k}-\frac{\partial f}{\partial y_k}
\frac{\partial g}{\partial x_k}\right)=\sum_{1\leq i,j\leq n}\Lambda^{ij}{\partial_{u_i}}f
{\partial_{u_j}}g \end{equation*}
(the Poisson brackets) and, more generally, for $l\geq 2$,
\begin{equation*}P^l(f,g):=\sum_{1\leq i_1,\ldots, i_l,j_1,\ldots,j_l\leq n}
\Lambda^{i_1j_1}\Lambda^{i_2j_2}\cdots \Lambda^{i_lj_l}\partial^l_{u_{i_1}\ldots u_{i_l}}f \,\partial^l_{u_{j_1}\ldots u_{j_l}}g.\end{equation*}
Then we have
\begin{equation} \label{eq:expM}
f \ast_M g:=\sum_{l\geq 0}\frac{1}{l!}\left(-\frac{i}{2}\right)^lP^l(f,g)\end{equation}
for each polynomials $f,g$ on ${\mathbb R}^{2n}$.

Note that we can use Equation \ref{eq:expM} as well as Equation \ref{eq:M} to extend $\ast_M$ to functions in $C^{\infty}({\mathbb R}^{2n})$ which are not necessarily polynomials
\cite{Vo}.

Similarly, we can define an associative product $\ast_1$ on functions on ${\mathbb R}^{2n}$
via
\begin{equation*}{W}_1^{-1}(f\ast_1 g)={W}_1^{-1}(f){W}_1^{-1}(g).\end{equation*}

For each function $f$ on ${\mathbb R}^{2n}$, we define the functions $f_{\lambda}$ and
$f^{\lambda}$ by $f_{\lambda}(x,y):=f(x,\lambda y)$ and $f^{\lambda}(x,y):=f(x,\tfrac{1}{\lambda}y)$.

Recall that ${\mathcal W}(f)=W_1^{-1}(f_{\lambda})$ for each (suitable) function $f$ on 
${\mathbb R}^{2n}$, see \cite{CaJLT, GBliv} and also Section \ref{sec:3}. This implies that
\begin{equation*}W_1^{-1}(f)W_1^{-1}(g)={\mathcal W}(f^{\lambda})
{\mathcal W}(g^{\lambda})={\mathcal W}(f^{\lambda}\ast_M g^{\lambda})=
W_1^{-1}((f^{\lambda}\ast_M g^{\lambda})_{\lambda})\end{equation*}
hence we have  $f\ast_1 g=(f^{\lambda}\ast_M g^{\lambda})_{\lambda}$ and 
Equation \ref{eq:expM} leads to the expansion
\begin{equation*} 
f \ast_1 g:=\sum_{l\geq 0}\frac{1}{l!}\left(-\frac{i}{2\lambda}\right)^lP^l(f,g).\end{equation*}

We can also consider the associative product $\ast_0$ associated with $W_0$ via
\begin{equation*}{W}_0^{-1}(f\ast_0 g)={W}_0^{-1}(f){W}_0^{-1}(g)\end{equation*}
for $f,g$ functions on ${\mathbb C}^n$.

Recall that $j:{\mathbb R}^{2n}\rightarrow {\mathbb C}^n$ is defined by $j(x,y)=x+iy$.
From the property \begin{equation*}W_1(A)=W_0({\mathcal B}A
{\mathcal B}^{-1})\circ j \end{equation*} for $A$ operator on $L^2({\mathbb R}^{n})$ (see Section \ref{sec:3}), we deduce that
\begin{equation*}f\ast_0 g=((f\circ j)\ast_1 (g\circ j))\circ j^{-1}=
\sum_{l\geq 0}\frac{1}{l!}\left(-\frac{i}{2\lambda}\right)^lP^l(f\circ j,g\circ j)\circ j^{-1}\end{equation*}
for $f, g$ functions on ${\mathbb C}^n$.

\subsection{Application to the star product of some Gaussians} \label{sub:92}
As a particular case of Proposition \ref{propWpi} we have that
\begin{align*}W_0(\sigma(t))(z)&=2^n\chi(t) 
\prod_{k=1}^n(1+e^{-i\alpha_k(t)})^{-1} \exp \left(-{\lambda}\vert z\vert^2 \right) \exp \left( 2{\lambda}
\sum_{k=1}^n(1+e^{i\alpha_k(t)})^{-1}\vert z_k\vert^2\right)\\
&=2^n\chi(t) 
\prod_{k=1}^n(1+e^{-i\alpha_k(t)})^{-1}\exp \left(-i\lambda \sum_{k=1}^n \vert z_k\vert^2
\tan \left(\tfrac{1}{2}\alpha_k(t)\right)\right),
\end{align*} 
for each $t\in {\mathbb R}^{m}$, since
\begin{equation*} 2(1+e^{i\alpha_k(t)})^{-1}-1=-i\tan \left(\tfrac{1}{2}\alpha_k(t)\right)\end{equation*} for $k=1,2,\ldots,n$.

Let  $t, t'\in {\mathbb R}^{m}$. We express the relation $\sigma(t+t')=\sigma(t)\sigma(t')$ in terms of the product $\ast_0$, that is, we write
\begin{equation*}W_0(\sigma(t+t'))=W_0(\sigma(t)\sigma(t'))=W_0(\sigma(t))\ast_0
W_0(\sigma(t'))\end{equation*} whenever the functions $W_0(\sigma(t)$, $W_0(\sigma(t')$
and $W_0(\sigma(t+t')$ are well-defined. This gives
\begin{align*}\exp &\left(-i\lambda \sum_{k=1}^n \vert z_k\vert^2
\tan \left(\tfrac{1}{2}\alpha_k(t)\right)\right)\ast_0\exp \left(-i\lambda \sum_{k=1}^n \vert z_k\vert^2
\tan \left(\tfrac{1}{2}\alpha_k(t')\right)\right)\\
=&2^{-n}\prod_{k=1}^n\frac{(1+e^{-i\alpha_k(t)})(1+e^{-i\alpha_k(t')}) }{(1+e^{-i\alpha_k(t+t')})}\exp \left(-i\lambda \sum_{k=1}^n \vert z_k\vert^2
\tan \left(\tfrac{1}{2}\alpha_k(t+t')\right)\right)\\
=&\prod_{k=1}^n \bigl(1-\tan \left(\tfrac{1}{2}\alpha_k(t)\right)\tan \left(\tfrac{1}{2}\alpha_k(t')\right)\bigr)^{-1}
\exp \left(-i\lambda \sum_{k=1}^n \vert z_k\vert^2
\tan \left(\tfrac{1}{2}\alpha_k(t+t')\right)\right).
\end{align*} 
Denoting $u_k:=\tan \left(\tfrac{1}{2}\alpha_k(t)\right)$ and $v_k:=\tan \left(\tfrac{1}{2}\alpha_k(t')\right)$ for $k=1,2,\ldots,n$, we can reformulate this relation as
\begin{align*}\exp \left(-i\lambda \sum_{k=1}^n u_k\vert z_k\vert^2
\right)&\ast_0\exp \left(-i\lambda \sum_{k=1}^nv_k \vert z_k\vert^2
\right)\\=&\prod_{k=1}^n(1-u_kv_k)^{-1}\,\exp \left(-i\lambda \sum_{k=1}^n\frac{u_k+v_k}{1-u_kv_k} \vert z_k\vert^2 \right).\end{align*}
In particular, taking $n=1$ and $\lambda =1$, we get the relation
\begin{equation*}\exp \left(-i  u\vert z\vert^2
\right)\ast_0\exp \left(-i v \vert z\vert^2
\right)=\frac{1}{1-uv}\,\exp \left(-i \frac{u+v}{1-uv} \vert z \vert^2 \right).\end{equation*}
Moreover, by changing $u$ to $-iu$ and $v$ to $-iv$ in this relation, we obtain
\begin{equation*}\exp \left( -u\vert z\vert^2
\right)\ast_0\exp \left(-v \vert z\vert^2
\right)\\=\frac{1}{1+uv}\,\exp \left( -\frac{u+v}{1+uv} \vert z\vert^2 \right)\end{equation*}
or, equivalently,
\begin{equation*}\exp \left( -u(x^2+y^2)
\right)\ast_M\exp \left(-v (x^2+y^2)
\right)=\frac{1}{1+uv}\,\exp \left( -\frac{u+v}{1+uv} (x^2+y^2) \right).\end{equation*}
This last relation is well known, see for instance \cite{CZa, CFZa, Der}.

\subsection{Application to the star exponential of polynomials} \label{sub:93} 
An important problem in Deformation Quantization is the computation of the star exponentials.
Consider, for instance the product $\ast_0$. Then the star exponential of a function $f$
on ${\mathbb C}^{n}$ is given by
\begin{equation*}\exp_{\ast_0}(f):=\sum_{k\geq 0}\,\frac{1}{k!}f^{\ast_0, k}\end{equation*}
where $f^{\ast_0, k}=f\ast_0\ldots \ast_0 f$ ($k$ times) for $k\geq 0$.

Usually, the computation of the star exponential of certain functions $f$ is performed by solving some differential system, see \cite{Arn, BM, BP}. Here we shall use the relation
\begin{equation*}W_0(\pi(\exp(X)))=W_0(\exp(d\pi(X)))=\exp_{\ast_0}(W_0(d\pi(X)))\end{equation*} for $X\in {\mathfrak g}$ together with Proposition \ref{propWpi} and
Proposition \ref{propWdpi} in order to obtain some closed formulas for the star exponential
(for $\ast_0$ and for the Moyal product) of  certain polynomials of degree $\leq 2$.

\begin{lemma} \label{lemexp} Let $X=(t,u,c)\in {\mathfrak g}$. Then, for each $s\in {\mathbb R}$, we have $\exp(sX)=(st,z(s),c(s))$ where $z(s)=(z_1(s),z_2(s),\ldots,z_n(s))$
and $c(s)$ are defined by 
\begin{equation*}z_k(s)=\frac{e^{i\alpha_k(t)s}-1}{i\alpha_k(t)}\,u_k, \quad k=1,2,\ldots,n
\end{equation*} and
\begin{equation*}c(s)=sc+\tfrac{1}{2}\sum_{k=1}^n \vert u_k\vert^2 \frac{\alpha_k(t)s-
\sin(\alpha_k(t)s)}{\alpha_k(t)^2}.\end{equation*}
\end{lemma}

\begin{proof} Let $X=(t,u,c)\in {\mathfrak g}$. Write  $\exp(sX)=(t(s),z(s),c(s))$. Then the relation $\exp((s_1+s_2)X)=\exp(s_1X)\exp(s_2X)$ for $s_1,s_2\in {\mathbb R}$ gives the following functional equations for the functions $t(s),z(s)$ and $c(s)$
\begin{equation*} \left\{
\begin{aligned}
&t(s_1+s_2)=t(s_1)+t(s_2),\\
&z(s_1+s_2)=z(s_1)+t(s_1)\cdot z(s_2),\\
&c(s_1+s_2)=c(s_1)+c(s_2)+\tfrac{1}{2}\omega( (z(s_1),\overline{z(s_1)}), (t(s_1)\cdot z(s_2),\overline{t(s_1)\cdot z(s_2)})).\\
\end{aligned}
\right. \end{equation*}
The first equation of the system gives $t(s)=st$. By differentiating the second equation at
$s_1=0$, we get
\begin{equation*}z'_k(s)=u_k+i\alpha_k(t)z_k(s),\quad k=1,2,\ldots,n.
\end{equation*}
Such differential equations are easy to solve and we find the announced formula for $z_k(s)$
for $k=1,2,\ldots,n$. Finally, by differentiating the third equation at $s_1=0$, we have
\begin{align*} c'(s)=&c+\tfrac{1}{4}i(u\overline{z(s)}-{\bar u}z(s))\\
=&c+\tfrac{1}{2}\sum_{k=1}^n \vert u_k\vert^2 \frac{1-
\cos(\alpha_k(t)s)}{\alpha_k(t)}.
\end{align*}
By integrating this last equation, we obtain the desired formula for $c(s)$.
\end{proof}

Now, we can reformulate Proposition \ref{propWpi} as follows.
\begin{proposition} \label{propWexp} Let $X=(t,u,c)\in {\mathfrak g}$. Then we have
\begin{align*}W_0&(\pi(\exp(X)))(z)=
2^n\chi(t) e^{i\lambda c} \prod_{k=1}^n(1+e^{-i\alpha_k(t)})^{-1}\\
\times & \exp \left(i\tfrac{\lambda}{2} \sum_{k=1}^n\vert u_k\vert^2
\frac{\alpha_k(t)-2\tan(\tfrac{1}{2}\alpha_k(t))}{\alpha_k(t)^2}\right)\\
\times & \exp\left(-i\lambda \sum_{k=1}^n\vert z_k\vert^2\tan\left(\tfrac{1}{2}\alpha_k(t)\right) \right)\\
\times &\exp \left(\lambda \sum_{k=1}^n\frac{1}{\alpha_k(t)}\tan\left(\tfrac{1}{2}\alpha_k(t)\right)
(z_k\bar{u_k}-\bar{z_k}u_k)\right).
\end{align*}
\end{proposition}

\begin{proof} We apply Proposition \ref{propWpi} with $g=\exp(X)$ for $X=(t,u,c)\in {\mathfrak g}$ using the expression of $\exp(X)$ given by Lemma \ref{lemexp}. The
obtained expression for $W_0(\exp(X))(z)$ is then simplified by some elementary calculations.
\end{proof}

\begin{corollary} \label{corW0} Let $a\in {\mathbb C}^{n}$, $c_0\in {\mathbb R}$, $b_k\in
{\mathbb R}$ for $k=1,2,\ldots,n$. Assume that $b_k\not= 0$ for each $k=1,2,\ldots,n$ and consider the polynomial
\begin{equation*}P(z)=ic_0+{\bar a}z-a{\bar z}+i\sum_{k=1}^n b_k\vert z_k\vert^2.
\end{equation*}
Then we have
\begin{align*} &\exp_{\ast_0} (P)(z)=e^{ic_0}\left(\prod_{k=1}^n\cos(\tfrac{1}{\lambda}b_k)\right)^{-1}\\
&\times \exp \left(i\lambda \sum_{k=1}^n \vert a_k\vert^2 \left(-\tfrac{1}{\lambda b_k}+\tfrac{1}{ b_k^2}\tan \left( \tfrac{1}{\lambda}b_k\right)\right)\right)\\
&\times\exp \left(i\lambda \sum_{k=1}^n \vert z_k\vert^2 \tan \left(\tfrac{1}{\lambda}b_k\right)\right)\\
&\times \exp \left( \lambda \sum_{k=1}^n \tfrac{1}{b_k} \tan \left( \tfrac{1}{\lambda}b_k\right)\left(z_k\bar{a_k}-a_k\bar{z_k}\right)\right).
\end{align*} 
In particular, if $a=0$ and $c_0=0$ then we obtain
\begin{equation*}\exp_{\ast_0} (i\sum_{k=1}^n b_k\vert z_k\vert^2)=\left(\prod_{k=1}^n\cos(\tfrac{1}{\lambda}b_k)\right)^{-1}\exp \left(i\lambda \sum_{k=1}^n \vert z_k\vert^2 \tan \left(\tfrac{1}{\lambda}b_k\right)\right).  \end{equation*}
\end{corollary}

\begin{proof} Recall that for each $X=(t,u,c)\in {\mathfrak g}$, we have
\begin{equation*}W_0(d\pi(X))(z)=d\chi(t)+i\lambda c+\tfrac{\lambda}{2}({\bar u}z-
{\bar z}u)+\tfrac{1}{2}i\sum_{k=1}^n\alpha_k(t)(1-\lambda \vert z_k\vert^2),
\end{equation*} see Proposition \ref{propWdpi}. We can take $\chi \equiv 1$ and choose 
$\alpha_k$, $k=1,2,\ldots,n$, and $X$ such that
\begin{equation*}\alpha_k(t)=-\tfrac{2}{\lambda}b_k, \,\, k=1,2,\ldots,n, \,\, u=\tfrac{2}{\lambda}a,\,\, c=\tfrac{1}{\lambda}c_0+\tfrac{1}{\lambda^2}
\sum_{k=1}^n b_k.\end{equation*}
Then we have $W_0(d\pi(X))=P$, hence 
\begin{equation*}\exp_{\ast_0}(P)=\exp_{\ast_0}(W_0(d\pi(X)))=W_0(\exp (d\pi(X)))=
W_0(\pi(\exp(X)))\end{equation*} and the result follows from Proposition \ref{propWexp}.
\end{proof} 

We can also formulate Corollary \ref{corW0} in terms of the Moyal product.

\begin{corollary} \label{corW1} Consider the polynomial 
\begin{equation*} P(x,y)=ic_0+2i(-vx+uy)+i\sum_{k=1}^n b_k(x_k^2+y_k^2)
\end{equation*} where $c_0\in {\mathbb R}$, $u,v\in {\mathbb R}^n$ and $b_k\in
{\mathbb R}$ for $k=1,2,\ldots,n$. Then we have
\begin{align*} &\exp_{\ast_M} (P)(x,y)=e^{ic_0}\left(\prod_{k=1}^n\cos(b_k)\right)^{-1}\\
&\times \exp \left( i\sum_{k=1}^n(u_k^2+v_k^2)\left(\tfrac{1}{b_k^2}\tan (b_k)-\tfrac{1}{b_k}\right)\right)\\
&\times \exp \left(i \sum_{k=1}^n(x_k^2+y_k^2)\tan (b_k)\right)\\
&\times \exp \left( 2i\sum_{k=1}^n\tfrac{1}{b_k}\tan (b_k)(y_ku_k-v_kx_k)\right).
\end{align*}
\end{corollary}

\begin{proof} Take $\lambda =1$. Then we have
\begin{equation*}(f\ast_0 g)\circ j=(f\circ j)\ast_M(g\circ j) \end{equation*}
for suitable functions $f, g$ on ${\mathbb C}^n$. Hence we have
\begin{equation*}\exp_{\ast_M}(P\circ j)=\exp_{\ast_0}(P)\circ j\end{equation*}
and we can apply Corollary \ref{corW0} with $a=u+iv$, $u,v\in {\mathbb R}^n$.
\end{proof}

Note that Corollary \ref{corW1} gives an expression for the Weyl symbol of the exponential of the operator on $L^2({\mathbb R}^n)$ whose Weyl symbol is the above polynomial $P(x,y)$. For more general results about arbitrary polynomials of degree $\leq 2$, see \cite{CaComp, Ho2}.

Note also that in \cite{BP}, the metaplectic representation of the non homogeneous symplectic group is constructed by using computations of star exponentials (for the Moyal product). It is, in some sense, the opposite process to the one we followed here.


\begin{thebibliography}{2}

\bibitem{Abd} Abdelmoula, L.,  Moment sets and unitary dual for the diamond group. Bull. Sci. Math. 134 (2010), no. 4, 379--390.

\bibitem{AU} Arazy, J. and Upmeier, H., Weyl Calculus for
Complex and Real Symmetric Domains, Harmonic analysis on complex
homogeneous domains and Lie groups (Rome, 2001). Atti Accad. Naz.
Lincei Cl. Sci. Fis. Mat. Natur. Rend. Lincei (9) Mat. Appl. 13, no
3-4 (2002), 165--181.

\bibitem{Arn} Arnal, D. (1988). The $\ast$-Exponential. In: Cahen, M., Flato, M. (eds) Quantum Theories and Geometry. Mathematical Physics Studies, vol 10, pp. 23-51. Springer, Dordrecht, 1988.

\bibitem{Barg} Bargmann, V., Group representations on Hilbert spaces of analytic functions. Analytic methods in mathematical physics (Sympos., Indiana Univ., Bloomington, Ind., 1968), pp. 27--63. Gordon and Breach, New York, 1970.

\bibitem{BM} Bayen, F. and Maillard, J.-M., Star exponentials of the elements of the inhomogeneous symplectic Lie algebra. Lett. Math. Phys. 6 (1982), 491--497.

\bibitem{Be1} Berezin, F. A., Quantization. Math. USSR Izv. 8, 5
(1974), 1109--1165.

\bibitem{Be2} Berezin, F. A., Quantization in complex symmetric
domains. Math. USSR Izv. 9, 2 (1975), 341--379.

\bibitem{Bern} Bernat, P., Conze, N., Duflo, M., L\'evy-Nahas, M., Ra\"is, M., Renouard, P. and Vergne, M. Repr\'esentations des groupes de Lie r\'esolubles. Monographies de la Soci\'et\'e Math\'ematique de France, No. 4. Dunod, Paris, 1972. 

\bibitem{BP} Burdet, G., Perrin, M., Weyl quantization and metaplectic representation. Lett. Math. Phys. 2 (1977/78), 93--99.

\bibitem{CaPad} Cahen, B., Berezin Quantization and Holomorphic Representations.
Rend. Sem. Mat. Univ. Padova 129 (2013), 277--297.

\bibitem{CaRiv} Cahen, B., Stratonovich-Weyl correspondence for
the diamond group. Riv. Mat. Univ. Parma 4 (2013), 197--213.

\bibitem{CaTo} Cahen, B., Weyl calculus on the Fock space and Stratonovich-Weyl correspondence for Heisenberg motion groups. Rend. Semin. Mat. Univ. Politec. Torino 76 (2018), 63-79.

\bibitem{CaJLT} Cahen, B., Stratonovich-Weyl correspondence for the generalized Poincar\'e group. J. Lie Theory 28 (2018), 1043--1062.

\bibitem{CaTs} Cahen, B., The complex Weyl calculus as a Stratonovich-Weyl correspondence for the diamond group. Tsukuba J. Math. 44 (2020), 121--137.

\bibitem{CaComp} Cahen, B.,  Complex Weyl symbols of metaplectic operators: an elementary approach. Rend. Istit. Mat. Univ. Trieste 55 (2023), Art. No. 5, 27 pp.

\bibitem{CaExt} Cahen, B., Complex Weyl symbols of the extended metaplectic representation operators. Oper. Matrices 18 (2024), no. 2, 457--477.

\bibitem{CR1} Combescure, M. and Robert, D., Coherent states and applications in mathematical physics. Theoretical and Mathematical Physics, Springer, Dordrecht, 2012.

\bibitem{CZa} Curtright, T. L. and Zachos, C. K., Quantum Mechanics in Phase Space.
Preprint, arXiv: 1104.5269v2[physics.hist-ph] (2011).

\bibitem{CFZa} Curtright, T. L., Fairlie, D. B. and Zachos, C. K., Quantum mechanics in phase space. An overview with selected papers. World Scientific Series in 20th Century Physics, 34. World Scientific Publishing Co. Pte. Ltd., Hackensack, NJ, 2005. 

\bibitem{Der} Derezi\'nski, J. and Karczmarczyk, M.,
Quantization of Gaussians. Kurasov, Pavel et al. (eds), Analysis as a tool in mathematical physics. Birkhäuser. Oper. Theory: Adv. Appl. 276 (2020), 277-304.

\bibitem{Gos} de Gosson, M. A. Symplectic methods in harmonic analysis and in mathematical physics. Pseudo-Differential Operators. Theory and Applications, 7. Birkh\"auser-Springer Basel AG, Basel, 2011.

\bibitem{Fo} Folland, B.,  Harmonic Analysis in Phase Space.
Princeton Univ. Press, 1989.

\bibitem{GB} Gracia-Bond\`{i}a, J. M., Generalized Moyal quantization
on homogeneous symplectic spaces. Deformation theory and quantum
groups with applications to mathematical physics (Amherst, MA,
1990), 93--114, Contemp. Math., 134, Amer. Math. Soc., Providence,
RI, 1992.

\bibitem{GBliv}Gracia-Bond\`{i}a, J. M., V\`{a}rilly, J. C. and Figueroa, H., Elements of noncommutative geometry. Birkh\"auser Advanced Texts: Basler Lehrb\"ucher. [Birkh\"auser Advanced Texts: Basel Textbooks] Birkh\"auser Boston, Inc., Boston, MA, 2001. 

\bibitem{GVS} Gracia-Bond\`{i}a, J. M., V\`{a}rilly, J. C. and Schempp, W.,
The Moyal representation of quantum mechanics and special function theory.
Acta Appl. Math. 18 (1990), 225-250.

\bibitem{Ho1} H\"ormander, L., The analysis of linear partial
differential operators. Vol. 3, Section 18.5, Springer-Verlag,
Berlin, Heidelberg, New-York, 1985.

\bibitem{Ho2} H\"ormander, L., Symplectic classification of quadratic forms and general Mehler formulas. Math. Z. 219 (1995), 413-449.

\bibitem{Kir} Kirillov, A. A., Lectures on the orbit method. Graduate Studies in Mathematics 64, American Mathematical Society, Providence, RI, 2004.

\bibitem{KVe} Kashiwara, M. and Vergne, M.,  On the Segal-Shale-Weil Representations and Harmonic Polynomials. Inventiones Math. 44 (1978), 1-47.

\bibitem{Lud} Ludwig, J., Dual topology of diamond groups. J. Reine Angew. Math. 467 (1995), 67--87.

\bibitem{Luo} Luo, S., Polar decomposition and isometric
integral transforms. Int. Transf. Spec. Funct. 9, 4 (2000),
313--324.

\bibitem{Ne} Neeb, K-H., Holomorphy and Convexity in Lie Theory. de
Gruyter Expositions in Mathematics, Vol. 28, Walter de Gruyter,
Berlin, New-York 2000.

\bibitem{Ner}Neretin, Y. A., Lectures on Gaussian integral operators and classical groups. EMS Series of Lectures in Mathematics. European Mathematical Society (EMS), Z\"urich, 2011.

\bibitem{PS} Pravda-Starov, K., Generalized Mehler formula for time-dependent non-selfadjoint quadratic operators and propagation of singularities. Math. Ann. 372 (2018), 1335--1382.

\bibitem{St} Stratonovich, R. L., On distributions in representation space.
Soviet Physics. JETP 4 (1957), 891--898.

\bibitem{Strea} Streater, R. F., The representations of the oscillator group. Comm. Math. Phys. 4 (1967), 217--236. 

\bibitem{Tay} M. E. Taylor, Noncommutative Harmonic Analysis. Mathematical
Surveys and Monographs 22, American Mathematical Society,
Providence, Rhode Island 1986.

\bibitem{UU} Unterberger, A. and Upmeier, H., Berezin transform
and invariant differential operators. Commun. Math. Phys. 164, 3
(1994), 563--597.

\bibitem{Vo} Voros, A., An Algebra of Pseudo differential operators and the
Asymptotics of Quantum Mechanics. J. Funct. Anal. 29 (1978),
104--132.







\end{thebibliography}
\end{document}